\documentclass[10pt,a4paper]{article}
\usepackage{amsmath,amsthm,amssymb,amscd}
\usepackage{graphicx,graphics}
\usepackage{verbatim}
\usepackage{mathrsfs}
\usepackage{enumerate}
\usepackage{color}
\usepackage[top=1.2in, bottom=1.2in, left=1.2in, right=1.2in]{geometry}

\newtheorem{theorem}{Theorem}[section]
\newtheorem{proposition}[theorem]{Proposition}

\newtheorem{problem}[theorem]{Problem}

\newtheorem{corollary}[theorem]{Corollary}

\newtheorem{conjecture}[theorem]{Conjecture}

\begin{document}

\title{Cores, joins and the Fano-flow conjectures}

\author{Ligang Jin\thanks{
       Paderborn Institute for Advanced Studies in
		Computer Science and Engineering,
		Paderborn University,
		Warburger Str. 100,
		33102 Paderborn,
		Germany;
       Supported by Deutsche Forschungsgemeinschaft (DFG) grant STE 792/2-1;
       ligang@mail.upb.de},
       Giuseppe Mazzuoccolo\thanks{
       Universit\`{a} di Modena e Reggio Emilia,
        Dipartimento di Scienze Fisiche,
        Informatiche e Matematiche,
        Via Campi 213/b,
        41125 Modena,
        Italy;
       Research performed within the project PRIN 2012 ``Strutture Geometriche, Combinatoria e loro Applicazioni'' of the Italian Ministry MIUR; mazzuoccolo.giuseppe@unimore.it},
       Eckhard Steffen\thanks{
       Paderborn Institute for Advanced Studies in
		Computer Science and Engineering,
		Paderborn University,
		Warburger Str. 100,
		33102 Paderborn,
		Germany;
       es@upb.de}}

\maketitle

\begin{abstract}
{\small The Fan-Raspaud Conjecture states that every bridgeless cubic graph has three 1-factors with empty intersection.
A weaker one than this conjecture is that every bridgeless cubic graph has two 1-factors and one join with empty intersection.
Both of these two conjectures can be related to conjectures on Fano-flows.
In this paper, we show that these two conjectures are equivalent to some statements on cores and weak cores of a bridgeless cubic graph.
In particular, we prove that the Fan-Raspaud Conjecture is equivalent to a conjecture proposed in [E. Steffen, 1-factor and cycle covers of cubic graphs, J. Graph Theory 78 (2015) 195-206].
Furthermore, we  disprove a conjecture proposed in [G. Mazzuoccolo, New conjectures on perfect matchings in cubic graphs, Electron. Notes Discrete Math. 40 (2013) 235-238] and we propose a new version of it under a stronger connectivity assumption.
The weak oddness of a cubic graph $G$ is the minimum number of odd components in the complement of a join of $G$.
We obtain an upper bound of weak oddness in terms of weak cores, and thus an upper bound of oddness in terms of cores as a by-product.}
\end{abstract}

\section{Introduction}
We study 1-factors (i.e., perfect matchings) in cubic graphs. If $G$ is a graph, then $V(G)$ and $E(G)$ denote its vertex set and edge set, respectively.
In 1994, the following statement was conjectured to be true by Fan and Raspaud:
\begin{conjecture} [\cite{FanRaspaud1994133}] \label{conj_FanRaspaud}
Every bridgeless cubic graph has three 1-factors $M_1, M_2, M_3$ such that $M_1\cap M_2\cap M_3 = \emptyset$.
\end{conjecture}
We remark that Conjecture \ref{conj_FanRaspaud} is implied by the celebrated Berge-Fulkerson Conjecture \cite{Fulkerson1971168}, which states
that every bridgeless cubic graph has six 1-factors such that each edge is contained in precisely two of them.

The study of Conjecture \ref{conj_FanRaspaud} leads to a deep analysis of Fano-flows on graphs. Consider the Fano plane $\cal{F}$$_7$ with the points labeled with the seven non-zero elements of $\mathbb{Z}^3_2$ as drawn in Figure \ref{fig_Fano}.
\begin{figure}[ht]
  \centering
 \includegraphics[width=5cm]{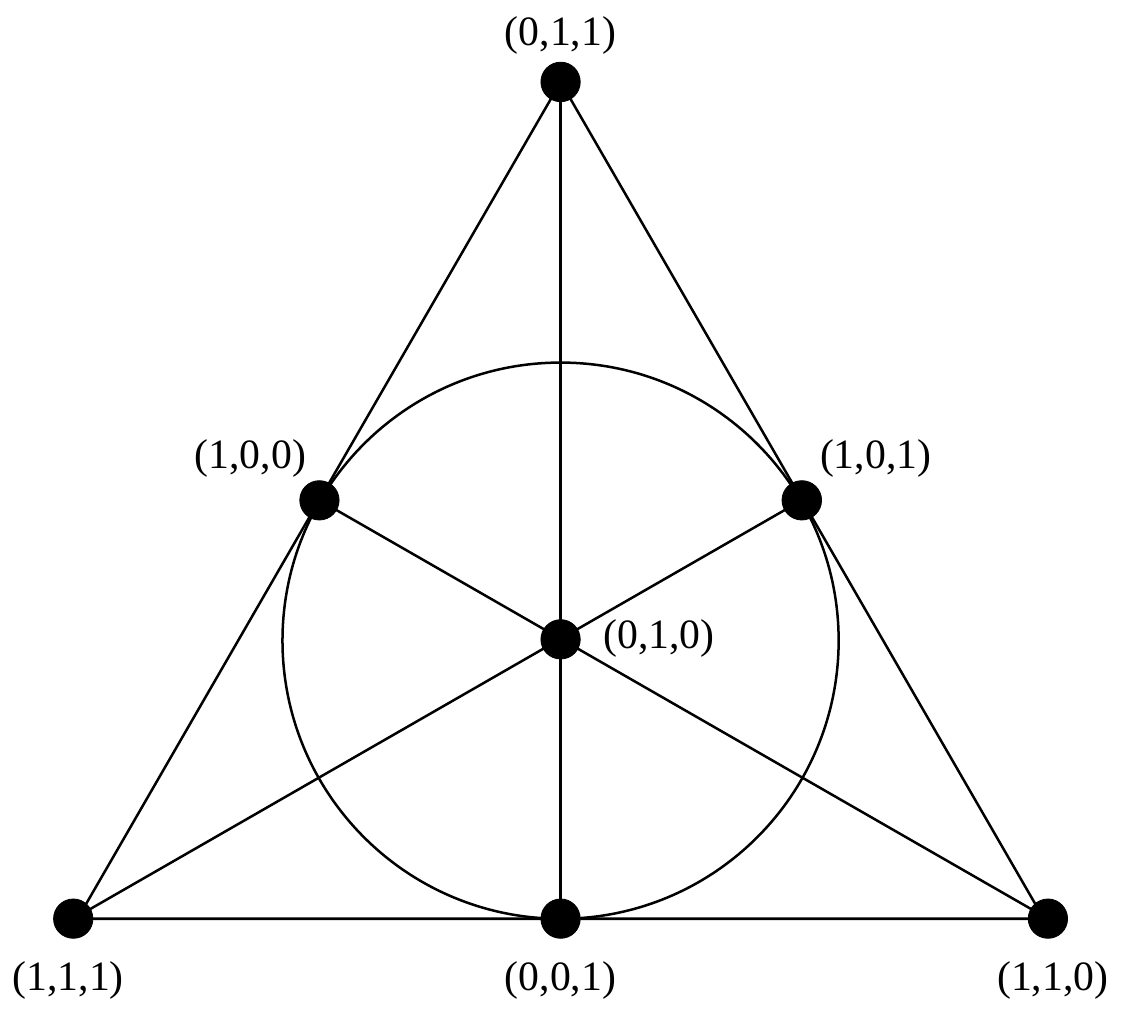}\\
  \caption{Fano plane $\cal{F}$$_7$}\label{fig_Fano}
\end{figure}
Clearly, if a cubic graph has a nowhere-zero $\mathbb{Z}^3_2$-flow, then for every vertex the flow values on its adjacent edges are pairwise different and they lie on a line of the Fano plane. Thus, every bridgeless cubic graph has a nowhere-zero Fano-flow by Jaeger's 8-flow Theorem \cite{Jaeger1979205}. However, it might be that not all possible combinations of three elements of $\mathbb{Z}^3_2$ appear at a vertex of $G$. For $k\leq 7$, a $k$-line Fano-flow is a Fano-flow of $G$ where at most $k$ lines of $\cal{F}$$_7$ appear as flow values at the vertices of $G$. Clearly, a 3-edge-colorable cubic graph has a 1-line Fano-flow. M\'{a}\v{c}ajov\'{a} and \v{S}koviera \cite{MacajovaSkoviera2005112} proved that each Fano-flow of a bridgeless cubic class 2 graph needs all seven points and at least four lines of the Fano plane. Furthermore, they proved that every bridgeless cubic graph has a 6-line Fano-flow, and conjectured that 4 lines are sufficient.

\begin{conjecture} [\cite{MacajovaSkoviera2005112}] \label{conj_4Fano}

Every bridgeless cubic graph has a 4-line Fano-flow.
\end{conjecture}
A natural relaxation of Conjecture \ref{conj_4Fano} is the following conjecture:

\begin{conjecture} [\cite{MacajovaSkoviera2005112}] \label{conj_5Fano}
Every bridgeless cubic graph has a 5-line Fano-flow.
\end{conjecture}

Let $H$ be a graph. If either $X\subseteq V(H)$ or $X\subseteq E(H)$, then $H[X]$ denotes the subgraph of $H$ induced by $X$.
A join of $H$ is a set $J$ of edges such that the degrees of every vertex have the same parity in $H$ and $H[J]$. If there is no harm of
confusion we use $J$ instead of $H[J]$.

Conjectures \ref{conj_4Fano} and \ref{conj_5Fano} have surprisingly counterparts in terms of 1-factors.
M\'{a}\v{c}ajov\'{a} and \v{S}koviera \cite{MacajovaSkoviera2005112} proved that Conjecture \ref{conj_4Fano} is equivalent to Conjecture \ref{conj_FanRaspaud}.
Analogously, one can easily obtain the equivalence between Conjecture \ref{conj_5Fano} and Conjecture \ref{conj_joinFR}, and the one between 6-line Fano-flow theorem and Proposition \ref{pro_1M2J}.

\begin{conjecture} \label{conj_joinFR}
Every bridgeless cubic graph has two 1-factors $M_1, M_2$ and a join
$J$ such that $M_1\cap M_2\cap J =\emptyset$.
\end{conjecture}

\begin{proposition} \label{pro_1M2J}
Every bridgeless cubic graph has a 1-factor $M$ and two joins
$J_1$ and $J_2$ such that $M\cap J_1\cap J_2 =\emptyset$.
\end{proposition}

Let $G$ be a bridgeless cubic graph.
The oddness $\omega(G)$ of $G$ is the minimum number of odd circuits of a 2-factor of $G$.
We define the \emph{weak oddness} $\omega'(G)$ of $G$ as the minimum number of odd components of the complement of a join.
Clearly, $\omega'(G) \leq \omega(G)$. Although  there is a long standing discussion on the question whether the converse is also true, we did not
find this question in any publication. We do not claim authorship but we think that this problem deserves study.

\begin{conjecture}
If $G$ is a bridgeless cubic graph, then $w(G)=w'(G)$.
\end{conjecture}

M\'{a}\v{c}ajov\'{a} and \v{S}koviera \cite{MacajovaSkoviera201461} proved
Conjecture \ref{conj_FanRaspaud} for cubic graphs with oddness at most 2. This implies the truth of Conjecture \ref{conj_joinFR} for these graphs as well.
A proof of this particular result is given in \cite{KaiserRaspaud20101307} by Kaiser and Raspaud. However, it is easy to see that $\omega'(G) = 2$ if and only if
$\omega(G) =2$, for each bridgeless cubic graph $G$. Hence, the result of \cite{MacajovaSkoviera201461} is even true for graphs with weak oddness at most 2.

Let $J$ be a join of a cubic graph $G$. Thus every vertex has degree either 1 or 3 in $J$.
A \emph{$J$-vertex} is a vertex of degree 3 in $J$.
Let $n(J)$ denote the number of $J$-vertices.

Let $G$ be a cubic graph and $S$ be a set of three joins $J_1,J_2$ and $J_3$ of $G$.
For $i\in \{0,\ldots,3\}$, let $E_i(S)$ be the set of edges that are contained in precisely $i$ elements of $S$.
When there is no harm of confusion, we write $E_i$ instead of $E_i(S)$.
The \emph{weak core} of $G$ with respect to $S$ (or to $J_1,J_2$ and $J_3$) is a subgraph $G_c$ induced by the union of sets $E_0, E_2$ and $E_3$, that is, $G_c=G[E_0\cup E_2\cup E_3]$.
The weak core $G_c$ is further called a \emph{$k$-weak $l$-core} where precisely $k$ elements of $S$ are not 1-factors and $l=|E_0|+\frac{3}{2}\sum_{i=1}^3n(J_i)$.
We define $\mu_3'(G) = \min \{l \colon\ \mbox{$G$ has a weak $l$-core}\}$.
A $0$-weak core is called a \emph{core} as well.
Define $\mu_3(G) = \min \{l \colon\ \mbox{$G$ has a $l$-core}\}$.
Clearly, $\mu_3'(G)\leq\mu_3(G)$.
It is easy to see that a bridgeless cubic graph $G$ is 3-edge-colorable if and only if $\mu_3'(G)=0.$

The core of a cubic graph was introduced by Steffen \cite{Steffen2014} working on perfect matching covers, and the parameter $\mu_3(G)$ was taken as a measurement on the edge-uncolorability of class 2 cubic graphs.
Weak cores are a natural generalization of the definition of cores for covers with three joins.
In this paper, we study both cores and weak cores of cubic graphs.

A join $J$ of $G$ is \emph{simple} if the subgraph induced by all the $J$-vertices contains no circuit. Clearly, every 1-factor of $G$ is a simple join, and every join of $G$ contains a simple join as a subset.
A weak core $G_c$ of $G$ is \emph{simple} if all the joins with respect to $G_c$ are simple.
A weak core is cyclic if it is a cycle.

Conjecture \ref{conj_FanRaspaud} can be easily formulated as a conjecture on cores in bridgeless cubic graphs:

\begin{conjecture}[\cite{Steffen2014}] \label{conj_cycle core}
Every bridgeless cubic graph has a cyclic core.
\end{conjecture}

Steffen proposed the following seemingly weaker conjecture:

\begin{conjecture}[\cite{Steffen2014}] \label{conj_bipartite core}
Every bridgeless cubic graph has a bipartite core.
\end{conjecture}
It is clear that Conjecture \ref{conj_cycle core} implies Conjecture \ref{conj_bipartite core} because all circuits in a cyclic core are of even length. Here, we show that the converse implication is also true, that is, Conjectures \ref{conj_cycle core} and \ref{conj_bipartite core} are equivalent. Hence, our result furnishes a new equivalent formulation for Fan-Raspaud Conjecture. We also show that the condition on the cores can be further relaxed. We even show that the following conjectures are
equivalent to Fan-Raspaud Conjecture.

\begin{conjecture}\label{conj_trianglefree}
Every bridgeless cubic graph has a triangle-free core.
\end{conjecture}

\begin{conjecture}\label{conj_3acyclic}
Every bridgeless cubic graph has three 1-factors such that the complement of their union is an acyclic graph.
\end{conjecture}

Analogously, we formulate Conjecture \ref{conj_joinFR} as a conjecture on 1-weak core:

\begin{conjecture}\label{conj_cycle star core}
Every bridgeless cubic graph has a cyclic 1-weak core.
\end{conjecture}

We prove the equivalence between this conjecture and the statement that every bridgeless cubic graph has a triangle-free simple 1-weak core.

In general, the Fano-flow can be related to cyclic weak core. Instead of $k$-line Fano-flow problem, we ask following equivalent question:

\begin{problem}
What is the minimum $k$ such that every bridgeless cubic graph has a cyclic $k$-weak core?
\end{problem}

As above, it was proved that $k<2$ and conjectured that either $k=0$ or $k=1$.

We summarize all announced implications in Figure \ref{summary}.

\begin{figure}[ht]
  \centering
 \includegraphics[width=10.5cm]{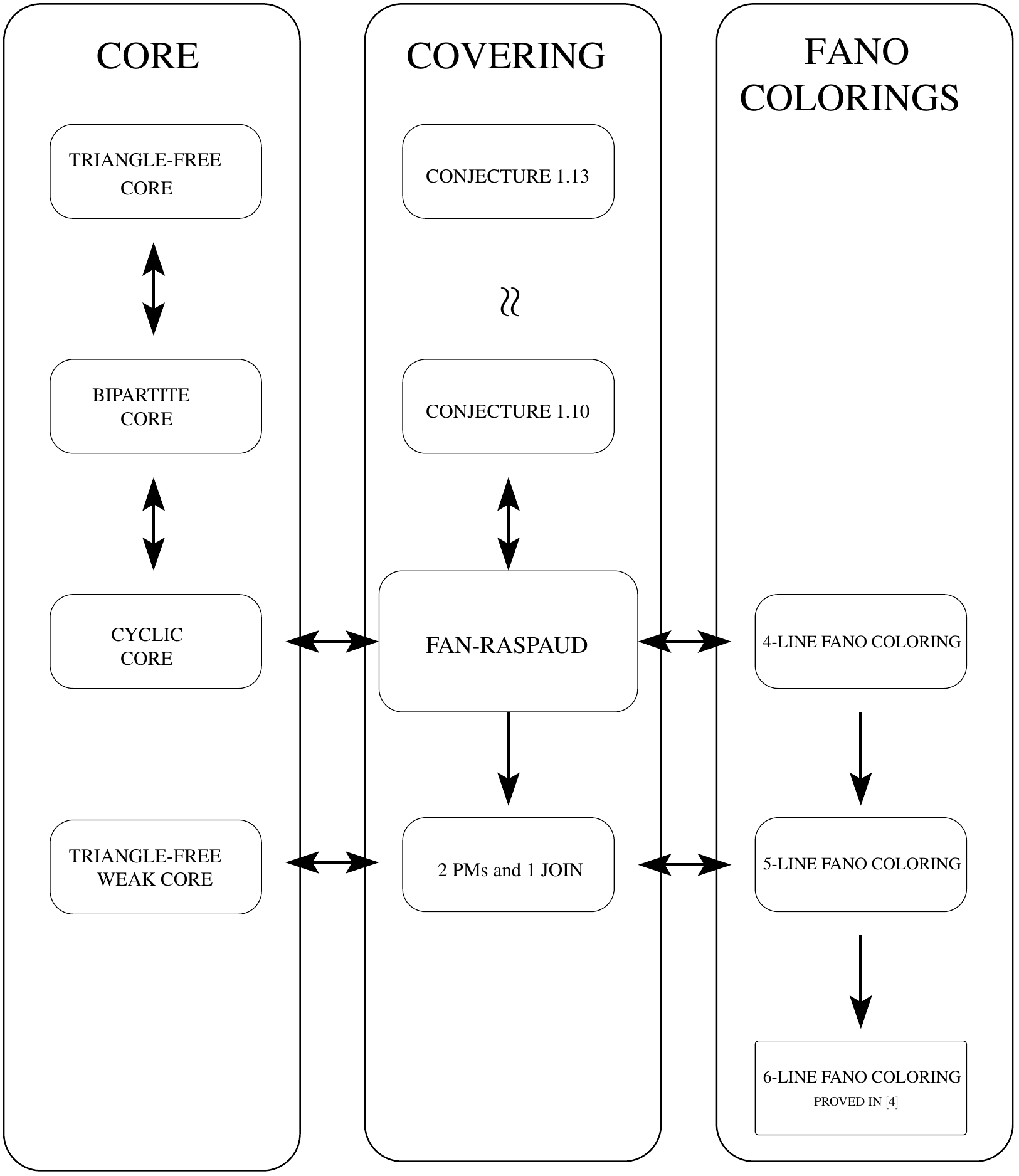}
  \caption{Conjectures related to Fan-Raspaud Conjecture}\label{summary}
\end{figure}

Section \ref{prop_weak_core} studies properties of weak cores, and it shows that the weak oddness of a bridgeless cubic graph is bounded in terms of its weak cores.

Finally, in the last section, we disprove the following stronger version of Conjecture \ref{conj_3acyclic}:
\begin{conjecture}[\cite{Mazzuoccolo2013235}] \label{conj_2acyclic}
Every bridgeless cubic graph has two 1-factors such that the complement of their union is an acyclic graph.
\end{conjecture}

Even if we prove that previous conjecture is false in that general form, we believe that it could be still true under stronger connectivity assumptions. In particular, we recall that it was verified true for all snarks, hence cyclically $4$-edge-connected cubic graphs, of order at most $34$ (see \cite{Mazzuoccolo2013235}).

More precisely, we wonder if every 3-connected (cyclically $4$-edge-connected) cubic graph has two 1-factors such that the complement of their union is an acyclic graph.

\section{The weak core of a cubic graph} \label{prop_weak_core}
Let $J_1,J_2$ and $J_3$ be three joins of a cubic graph $G$.
We say that a vertex $v$ of $G$ has type $(x,y,z)$ if the three edges incident to $v$ are covered $x,y$ and $z$ times by $\{J_1,J_2,J_3\}$, respectively. We denote by $a,b,c,d,e,f,g$ the number of vertices of type $(3,3,3),(3,2,2)$, $(3,1,1),(2,2,1)$, $(1,1,1),(2,1,0),(3,0,0)$, respectively (see also Figure \ref{fig_types}). Clearly, every vertex has precisely one type. Note that vertices of type $(3,3,3)$, $(3,2,2)$, $(3,1,1)$ and $(2,2,1)$ are $J_i$-vertices for some $i$.

\begin{figure}[ht]
  \centering
 \includegraphics[width=7cm]{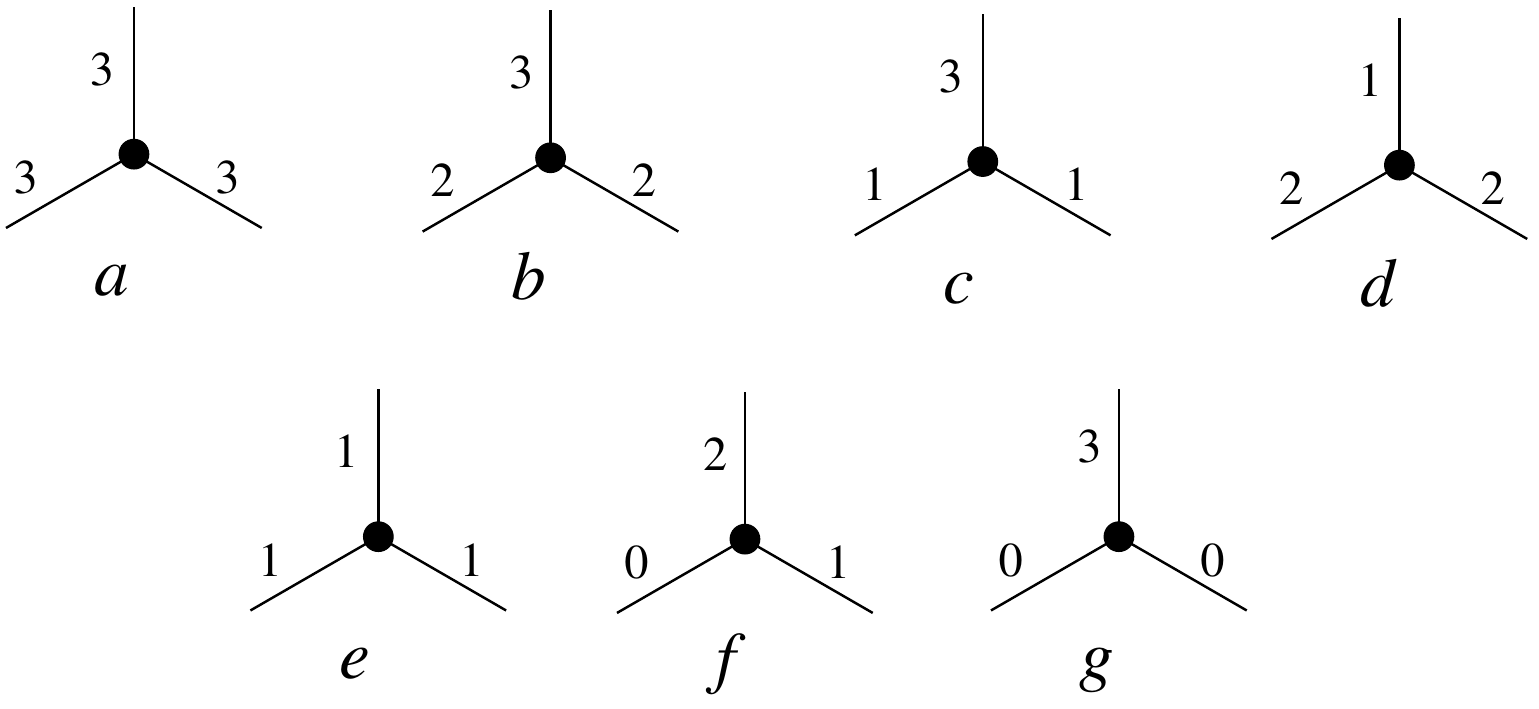}\\
  \caption{Vertex types}\label{fig_types}
\end{figure}

\begin{proposition}\label{lem_edges_equality}
Let $G$ be a cubic graph, and $J_1,J_2$ and $J_3$ be three joins of $G$. Then
$$|E_0|+\sum_{i=1}^3 n(J_i)=|E_2|+2|E_3|.$$
 \end{proposition}
\begin{proof}
By type definitions, we have $\sum_i n(J_i)=3a+2b+c+d$, $|E_0|=\frac{f}{2}+g$, $|E_2|=b+d+\frac{f}{2}$ and $|E_3|=\frac{3a}{2}+\frac{b}{2}+\frac{c}{2}+\frac{g}{2}$. Hence, $\sum_i n(J_i)+|E_0|=3a+2b+c+d+\frac{f}{2}+g=|E_2|+2|E_3|$ holds.
\end{proof}

\begin{proposition}\label{thm_structure of weak-core}
If $G_c$ is a weak core of a cubic graph $G$, then $G[E_0\cup E_2]$ is either an empty graph or a cycle.
\end{proposition}

\begin{proof}
By type definitions, it is easy to see that every vertex is incident with either none or precisely two edges of $E_0\cup E_2$.
Therefore, $G[E_0\cup E_2]$ is either an empty graph or a cycle.
\end{proof}

Let $H$ be a graph. We denote by $|H|_{odd}$ the number of odd components of $H$.
If $J$ is a join of $H$, then $\overline{J}$ denotes the complement of $J$.

\begin{theorem}\label{thm_oddness_core}
Let $G$ be a bridgeless cubic graph and $G_c$ be a weak $k$-core with respect to three joins $J_1, J_2$ and $J_3$.
Then $\sum_{i=1}^3|\overline{J_i}|_{odd} \leq k.$
\end{theorem}
\begin{proof}
Each component of the complement of $J_i$ is either an isolated vertex or a circuit.
Any odd circuit of $\overline{J_i}$ contains either one edge from $E_0$ or a $J_k$-vertex with $k\neq i$.
We call an odd circuit  of $\overline{J_i}$ \emph{bad} if it has no $J_k$-vertex for $k \neq i$.
In what follows we distinguish elements of $E_0$ according to their behavior with respect to bad circuits. We define:
$$X_i=\{e \colon\ \text{$e$ is the unique edge in } C \cap E_0 \text{ and } C \text { is a bad circuit of } \overline{J_i}   \}, \text{ for } i=1,2,3;$$
$$Y_i=\{e \colon e \in E_0 \setminus X_i \text{ and } e \in C \cap E_0 \text{ and } C \text { is a bad circuit of } \overline{J_i}   \}, \text{ for } i=1,2,3.$$
Set $x=|X_1|+|X_2|+|X_3|$ and $y=|Y_1|+|Y_2|+|Y_3|$.

Since $X_i \cap Y_i = \emptyset$, it follows that
\begin{equation}
x+y\leq 3|E_0|.
\end{equation}

Moreover, if $e \in X_i$, then $e \not \in X_j$, and $e \not \in X_k$ for $j,k \neq i$, that is
\begin{equation}
x\leq |E_0|.
\end{equation}
Combining equations (1) and (2) implies
\begin{equation}\label{eq_1}
x+\frac{y}{2}\leq 2|E_0|.
\end{equation}

Now, we are in position to prove our assertion.
Since in an odd circuit of $\overline{J_i}$ there is either a $J_k$-vertex ($k \neq i$) or an edge of $X_i$ or two edges of $Y_i$, the following relation holds:
$$|\overline{J_i}|_{odd} \leq |X_i|+\frac{|Y_i|}{2}+\sum_{i=1}^3 n(J_i).$$
Therefore, by summing up for all three joins we deduce:
$$\sum_{i=1}^3|\overline{J_i}|_{odd} \leq x+\frac{y}{2}+3\sum_{i=1}^{3}n(J_i) \leq 2|E_0|+3\sum_{i=1}^{3}n(J_i)=k,$$
where the last inequality directly follows from (\ref{eq_1}).
\end{proof}

\begin{corollary} \label{coro_womega-wmu3}
If $G$ is a bridgeless cubic graph, then $\omega'(G)\leq \frac{2}{3}\mu'_3(G).$
\end{corollary}

\begin{proof}
Let $G_c$ be a weak $\mu_3'(G)$-core of $G$ with respect to three joins $J_1,J_2$ and $J_3$.
By Theorem \ref{thm_oddness_core}, we have $|\overline{J_1}|_{odd}+|\overline{J_2}|_{odd}+|\overline{J_3}|_{odd}\leq 2\mu'_3(G)$.
By the minimality of the weak oddness $\omega'(G)$ it follows that $\omega'(G)\leq \frac{2}{3}\mu'_3(G)$ .
\end{proof}

The following results were already obtained in \cite{JinSteffen20150105}, but now it turns out that they are just a particular case of our previous theorem.
\begin{theorem}\label{thm_odd components}
Let $G$ be a bridgeless cubic graph and $G_c$ be a $k$-core with respect to three 1-factors $M_1, M_2$ and $M_3$.
Then $\sum_{i=1}^3|\overline{M_i}|_{odd} \leq k.$
\end{theorem}

\begin{corollary} \label{coro_womega-wmu3}
If $G$ is a bridgeless cubic graph, then $\omega(G)\leq \frac{2}{3}\mu_3(G).$
\end{corollary}

\begin{proof}
Let $G_c$ be a $\mu_3(G)$-core of $G$ with respect to three 1-factors $M_1,M_2$ and $M_3$.
By Theorem \ref{thm_odd components}, we have $|\overline{M_1}|_{odd}+|\overline{M_2}|_{odd}+|\overline{M_3}|_{odd}\leq 2\mu_3(G)$.
By the minimality of $\omega(G)$, it follows that $\omega(G)\leq \frac{2}{3}\mu_3(G)$.
\end{proof}

\section{Equivalent Statements}
Let $G_1$ and $G_2$ be two bridgeless graphs, $e_1$ and $e_2$ be two edges such that $e_1=u_1v_1\in E(G_1)$ and $e_2=u_2v_2\in E(G_2)$.
The \emph{2-cut connection} on $\{e_1,e_2\}$ is a graph operation that consists of deleting edges $e_1$ and $e_2$ and adding two new edges $u_1u_2$ and $v_1v_2$.
Clearly, the graph obtained from $G_1$ and $G_2$ by applying 2-cut connection is also bridgeless.

\begin{theorem}\label{thm_equivalence}
The following three statements are equivalent:
\begin{enumerate}[(1)]
  \item (Conjecture \ref{conj_5Fano}) Every bridgeless cubic graph has a 5-line Fano-flow.
  \item (Conjecture \ref{conj_joinFR}) Every bridgeless cubic graph has a join $J$ and two 1-factors $M_1$ and $M_2$ such that $J\cap M_1\cap M_2=\emptyset.$
  \item Every bridgeless cubic graph has a cyclic 1-weak core.
  \item Every bridgeless cubic graph has a triangle-free simple 1-weak core.
\end{enumerate}
\end{theorem}

\begin{proof} The equivalence of statements (1) and (2) follows easily from the results in \cite{MacajovaSkoviera2005112}.

(2) $\rightarrow$ (3): By Proposition \ref{thm_structure of weak-core}, the 1-weak core with respect to $M_1,M_2$ and $J$ is cyclic.
Therefore, statement (2) implies statement (3).

(3) $\rightarrow$ (4): Suppose to the contrary that there is a bridgeless cubic graph $G$ that has no triangle-free simple 1-weak core. By (3), $G$ has a cyclic 1-weak core, and let $G_c$ be a cyclic 1-weak core of $G$ with fewest edges. Let $G_c$ be with respect to two 1-factors $M_1,M_2$ and a join $J$.
We claim that $G_c$ is simple. Otherwise, $J$ is not simple, that is, $G$ contains a circuit $C$ such that each vertex of $C$ is a $J$-vertex. Recall that $G_c$ is cyclic, by type definitions according to $M_1,M_2$ and $J$, every vertex of $C$ has type $(2,2,1)$. Let $J_1$ be obtained from join $J$ by removing all the edges of $C$. Thus $J_1$ is also a join of $G$. The 1-weak core with respect to $M_1,M_2$ and $J_1$ is cyclic and has fewer edge than $G_c$, a contradiction. This completes the proof of the claim.

By our supposition and the previous claim, $G_c$ has a triangle $[xyz]$.
It follows that two of vertices $x,y$ and $z$ have type $(2,1,0)$ and the last one has type $(2,2,1)$, which is the only possible case.
Without loss of generality we assume that $z$ is of type $(2,2,1)$.
Set $J_2=J\cup \{xy\}\setminus \{xz,yz\}.$ Clearly, $J_2$ is a join of $G$.
Now the 1-weak core with respect to $M_1,M_2$ and $J_2$ is cyclic and has fewer edges than $G_c$, a contradiction.
Therefore, statement (3) implies statement (4).

(4) $\rightarrow$ (2): Let $G$ be a bridgeless cubic graph with edge set $\{e_1,\ldots,e_m\}$.
Take $m$ copies $T_1,\ldots, T_m$ of the complete graph $K_4$.
For $i\in \{1,\ldots,m\}$, apply the 2-cut connection on $e_i$ and an edge of $T_i$, and let $e_i'$ and $e_i''$ be the two added new edges.
The resulting graph $G'$ is bridgeless and cubic.
By (4), $G'$ has a triangle-free simple 1-weak core $H$.
Let $H$ be with respect to two 1-factors $M_1,M_2$ and a simple join $J$.
For every join $F$ of $G'$, since $F$ contains either both of $e'_i$ and $e''_i$ or none of them for $i\in \{1,\dots,m\}$,
let $con(F)=\{e\colon\ e=e_i\in E(G)$, and $e_i',e_i''\in F\}$. Clearly, $con(F)$ is a join of $G$ and in particular, $con(F)$ is a 1-factor of $G$ if $F$ is a 1-factor of $G'$. We claim that $con(M_1)\cap con(M_2)\cap con(J)=\emptyset$ and hence, statement (2) holds.
Suppose to the contrary that $G$ has an edge $e_1$ contained in all of $con(M_1), con(M_2)$ and $con(J)$. It follows that $e_1',e_1''\in M_1\cap M_2\cap J$, and hence
one can easily deduce that in copy $T_1$, the 1-weak core $H$ contains either a triangle or a circuit of length 4 whose vertices are $J$-vertices, a contradiction. Therefore, statement (4) implies statement (2).
\end{proof}

\begin{theorem}
The following four statements are equivalent:
\begin{enumerate}[(1)]
  \item (Conjecture \ref{conj_4Fano}) Every bridgeless cubic graph has a 4-line Fano-flow.
  \item (Conjecture \ref{conj_FanRaspaud}) Every bridgeless cubic graph has three 1-factors $M_1,M_2,M_3$ such that $M_1\cap M_2\cap M_3=\emptyset.$
  \item (Conjecture \ref{conj_bipartite core}) Every bridgeless cubic graph has a bipartite core.
  \item (Conjecture \ref{conj_trianglefree}) Every bridgeless cubic graph has a triangle-free core.
  \item (Conjecture \ref{conj_3acyclic}) Every bridgeless cubic graph has three 1-factors such that the complement of their union is an acyclic graph.
\end{enumerate}
\end{theorem}

\begin{proof}
 The equivalence of statements (1) and (2) is proved in \cite{MacajovaSkoviera2005112}.

If statement (2) holds, then by Proposition \ref{thm_structure of weak-core}, the core $G_c$ of a bridgeless cubic graph $G$ with respect to $M_1,M_2,M_3$ is cyclic. More precisely, each circuit in $G_c$ contains edges from $E_0$ and $E_2$ alternate in cyclic order. Hence, the core $G_c$ is bipartite and triangle-free, and $G[E_0]$ is an acyclic graph. Hence statement (2) implies all of the statements (3), (4) and (5).

Let $G$ be a bridgeless cubic graph with edge set $\{e_1,\ldots,e_m\}$.
Take $m$ copies $T_1,\ldots, T_m$ of the complete graph $K_4$.
For $i\in \{1,\ldots,m\}$, apply 2-cut connection on $e_i$ and an edge of $T_i$, and let $e_i'$ and $e_i''$ the two added new edges.
Let $G'$ be the resulting graph, which is bridgeless and cubic.
Let $H$ be a core of $G'$ with respect to three 1-factors $M_1,M_2,M_3$.
For every 1-factor $F$ of $G'$, since $F$ contains either both of $e'_i$ and $e''_i$ or none of them for $i\in \{1,\dots,m\}$,
we can let $con(F)=\{e\colon\ e=e_i\in E(G)$, and $e_i',e_i''\in F\}$. Clearly, $con(F)$ is a 1-factor of $G$.
We claim that if $H$ is either bipartite or triangle-free or if the complement of the union of $M_1,M_2,M_3$ is acyclic, then $con(M_1),con(M_2)$ and $con(M_3)$ have empty intersection. This claim completes the proof.
Suppose to the contrary that $G$ has an edge $e_1$ such that $e_1\in con(M_1)\cap con(M_2)\cap con(M_3)$. It follows that $e_1',e_1''\in M_1\cap M_2\cap M_3$. Hence in copy $T_1$, core $H$ contains triangles and $G[E_0]$ contains a circuit of length 4, a contradiction with the supposition of our claim.
\end{proof}

\section{Counterexample to Conjecture \ref{conj_2acyclic}}

If the Fan-Raspaud Conjecture is true, then every bridgeless cubic graph has two 1-factors, say $M_1$ and $M_2$, with no odd edge-cut in their intersection; in particular, the complement of $M_1 \cup M_2$ is a bipartite graph which is union of paths and even circuits.
One could asks if even circuits could be forbidden in such bipartite graph. It is verified to be true for all snarks of order at most $34$ and proposed as a conjecture in \cite{Mazzuoccolo2013235}.

Here, we disprove Conjecture \ref{conj_2acyclic} in its present formulation by using the same technique already used in the proof of Theorem \ref{thm_equivalence}.

Let $P$ be the Petersen graph and let $\{e_1,\ldots,e_{15}\}$ be its edge-set. Take $15$ copies $T_1,\ldots,T_{15}$ of the complete graph $K_4$. For $i \in \{1,\ldots,15\}$, apply a $2$-cut connection on $e_i$ and an arbitrary edge of $T_i$. Denote by $G$ the graph obtained.
Let $M_1$ and $M_2$ be two 1-factors of $G$, and let $con(M_1)$ and $con(M_2)$ be the two corresponding 1-factors of $P$, respectively.
Since every pair of 1-factors of $P$ has exactly an edge in common, without loss of generality we can assume $\{e_1\}=con(M_1) \cap con(M_2)$. Hence, $T_1$ has an edge covered twice and a $4$-circuit uncovered, that is the complement of $M_1 \cup M_2$ is not acyclic.

We would like to stress that the graph $G$ has a lot of 2-edge-cuts, so we wonder if an analogous version of Conjecture \ref{conj_2acyclic} could hold true for $3$-connected or cyclically $4$-edge-connected cubic graphs.

\end{document}